\documentclass{gtpart}     
\usepackage{epsfig,color} 
\usepackage{mathtools,amsmath}

\title{Curve complexes and Garside groups}

\author{Matthieu Calvez}
\givenname{Matthieu}
\surname{Calvez}
\address{Matthieu Calvez, Departamento de matem\'{a}tica y ciencia de la computaci\'{o}n , Facultad de Ciencia, Universidad de Santiago de Chile, Av. Libertador B. O'Higgins 3363, Santiago, Chile }
\email{calvez.matthieu@gmail.com}

\author{Bert Wiest}
\givenname{Bert}
\surname{Wiest}
\address{Bert Wiest, UFR Math\'ematiques, Universit\'e de Rennes 1, 35042 Rennes Cedex, France}
\email{bertold.wiest@univ-rennes1.fr}


\volumenumber{}
\issuenumber{}
\publicationyear{}
\papernumber{}
\startpage{}
\endpage{}
\doi{}
\MR{}
\Zbl{}
\received{}
\revised{}
\accepted{}
\published{}
\publishedonline{}
\proposed{}
\seconded{}
\corresponding{}
\editor{}
\version{}

\makeautorefname{notation}{Notation}%

\newtheorem{theorem}{Theorem}[section]
\newtheorem{lemma}[theorem]{Lemma} 
\newtheorem{proposition}[theorem]{Proposition}

\newtheorem{conjecture}[theorem]{Conjecture} 
\theoremstyle{definition}
\newtheorem{definition}[theorem]{Definition}    
\newtheorem{remark}[theorem]{Remark}

\newtheorem{openproblems}[theorem]{Open Problems}\newtheorem{example}[theorem]{Example}

\newtheorem{notation}[theorem]{Notation}

\newtheorem{observation}[theorem]{Observation}
\def\Z{{\mathbb Z}}

\def\co{\colon \thinspace}

\def\MMmodel{\mathcal{CC}\hat{\phantom{I}}}
\renewcommand{\phi}{\varphi}

\begin{document}

\begin{abstract}
We present a simple construction which associates to every Garside group a metric space, called the additional length complex, on which the group acts. 
These spaces share important features with curve complexes: they are $\delta$-hyperbolic, infinite, and typically locally infinite graphs. 
We conjecture that, apart from obvious counterexamples, additional length complexes have always infinite diameter. We prove this conjecture for the classical example of braid groups $(B_n,B_n^{+},\Delta)$; moreover, in this framework, reducible and periodic braids act elliptically, and at least some pseudo-Anosov braids act loxodromically. We conjecture that for $B_n$, the additional length complex is actually quasi-isometric to the curve complex of the $n$~times punctured disk.
\end{abstract}

\maketitle


\section{Motivation}

Let us consider the braid group~$B_n$ acting on the left on the curve complex of the $n$~times punctured disk, which we denote $\mathcal{CC}$. This complex is equipped with a base point~$c_0$, which we take to be a round curve in the disk. There is an obvious map
$$
B_n \longrightarrow \mathcal{CC}, \ x\mapsto x.c_0
$$

Now, consider the classical Garside structure of the braid group: permutation braids (or simple braids) are chosen as a preferred set of generators.
For any element $x$~of~$B_n$, the Garside \emph{mixed normal form} (as defined in~\cite{Thurston}) gives rise to a path in the Cayley graph, which is actually a geodesic~\cite{Charney}. We shall look at the image of this geodesic in $\mathcal{CC}$ -- thus, if $x$~has Garside mixed normal form $x=x_1\cdot\ldots\cdot x_l$, we consider the path $c_0$, $x_1.c_0$, $x_1x_2.c_0, \ \ldots\ $, $x_1\cdot\ldots\cdot x_l.c_0$ in the curve complex. Not much is known about this family of paths. For instance, it is not known whether it forms a uniform family of unparametrized quasi-geodesics. We conjecture that this is true, but this is by no means obvious: 
it is definitely \emph{not} true that any quasi-geodesic in $B_n$ projects to an unparametrized quasi-geodesic in $\mathcal{CC}$~\cite{SchleimerWiest}.

Even if we suppose that this first conjecture is true, i.e.\ normal form words in~$B_n$ project to quasi-geodesics in~$\mathcal{CC}$, another question remains. Indeed, let us look at a triangle in the Cayley graph of $B_n$ with vertices $1_{B_n}$ and positive braids $x, y\in B_n^+$, and with edges the mixed normal forms of $x$, of $y$, and of $x^{-1}y$. Projecting this triangle to the curve complex as above, and assuming the first conjecture to be true, we must obtain a $\delta$-thin triangle (since the curve complex is Gromov-hyperbolic~\cite{MM1,HPW,PrzSisto}). 
Now the obvious question is: how can we characterise, in terms of the three normal forms, the position of the quasi-center (the point which is close to all three edges)? 

There is an obvious conjecture how to answer this question: the quasi-center should be at $(x\wedge y).c_0$, where $x\wedge y$ denotes the greatest common divisor of $x$ and $y$, in the sense of Garside theory~\cite{B-G-GM}.  
Moreover, the edges from $c_0$ to~$x.c_0$ and from $c_0$ to~$y.c_0$ should stay close to each other (and to the path from $c_0$ to $(x\wedge y).c_0$) up to length \ $length(x\wedge y)$, and diverge afterwards.  This is our second conjecture.

The aim of the present paper is not to prove either of the above two conjectures, but rather to show what happens if we ``squash down'' the Cayley graph of $B_n$ in such a way that the second conjecture is forced to hold. It turns out that the resulting space, which we call the "additional length complex'' $\mathcal C_{AL}$, is $\delta$-hyperbolic, and shares many properties with the curve complex - we conjecture that the two are actually quasi-isometric. 

What is remarkable is that our construction of $\mathcal C_{AL}$ does not actually mention curves on a surface, and can be carried out analogously for any finite type Garside structure of a finite type Garside group. Garside groups are a family of groups with good combinatorial and algorithmic properties, containing e.g.\ Artin groups of spherical type~\cite{DehornoyParis,DehornoyGarside,GarsideFoundations}. For a particularly readable introduction which contains almost all prerequisites for this paper, see~\cite[Section 1.1]{B-G-GM}. For the rest of the paper, whenever we talk about a Garside group, we mean a Garside group of finite type equipped with a specific Garside structure.

Thus any Garside group $G$ acts on a metric space $\mathcal{C}_{AL}(G)$. The results of this paper can be summarized as follows.

{\bf Theorem } {\sl (A) For any Garside group~$G$, the space $\mathcal C_{AL}(G)$ is $60$-hyperbolic. Moreover, normal form words in~$G$ give rise to paths in $\mathcal C_{AL}$ which are at distance at most 39 from geodesics connecting the endpoints.}

{\sl (B) If $G$ is the braid group $B_n$, equipped with the classical Garside structure, then $\mathcal C_{AL}$ is of infinite diameter. Moreover, periodic and reducible braids act elliptically, and there exists a pseudo-Anosov braid which acts loxodromically.}
\bigskip

The plan of the paper is as follows: in Section~\ref{S:MainResult}, after recalling a few basic facts about Garside groups, we construct the additional length complex and prove that it is $\delta$-hyperbolic. In Section~\ref{S:B_n} we prove that the additional length complex associated with the classical Garside structure of the braid group is moved elliptically by the action of periodic and reducible braids, and that it is of infinite diameter.


\section{The main result}\label{S:MainResult}

In this section we shall prove that every Garside group~$G$ acts on a $\delta$-hyperbolic complex which we call the \emph{additional length complex} of $G$ 
(\emph{complexe des longueurs suppl\'ementaires} in French). 
The key ingredient for proving hyperbolicity is a ``Guessing Geodesics Lemma'' of Bowditch~\cite{Bowditch}. The definition of the complex rests on the technical notion of \emph{absorbable element}; we start with the definition and first properties of those.


\subsection{Absorbable elements}\label{SS:Absorbable}

In what follows, $(G,P,\Delta)$ is a Garside group 
with positive monoid~$P$, Garside element~$\Delta$, and $\tau$~denotes the inner automorphism of $G$ given by $\tau(x)=\Delta^{-1}x\Delta$.
In particular $P$ is \emph{atomic}, i.e. it is generated by the set of elements $a\in P$ such that the relation $a=uv$ with $u,v\in P$ implies $u=1$ or $v=1$; these elements are called \emph{atoms}.  
We assume the reader to be familiar with the prefix and suffix orders $\preccurlyeq$ and $\succcurlyeq$, the left/right-weightedness, the left/right gcd ($\wedge$/$\wedge^{\hspace{-0.7mm}\Lsh\hspace{0.7mm}}$) and lcm ($\vee$/$\vee^{\hspace{-0.7mm}\Lsh\hspace{0.7mm}}$) and the left/right normal form -- see e.g.\ \cite[Section 1.1]{B-G-GM}.
We recall that to each element $x$ of $G$ are associated three relative integers: its infimum 
$\inf(x)=\max\{r \in \mathbb Z, \Delta^r\preccurlyeq x\}$, its supremum $\sup(x)=\min\{s\in \mathbb Z, x\preccurlyeq \Delta^r\}$ and its canonical length $\ell(x)=\sup(x)-\inf(x)$. These are related to the left normal form as follows: if $x$ has left normal form 
$x=\Delta^p x_1\ldots x_r$, $p$, $p+r$ and $r$ are the infimum, the supremum and the canonical length of $x$, respectively.

We also recall the notion of rigidity: an element $x$ of $G$ with left normal form $x=\Delta^p x_1\ldots x_r$ is said to be \emph{rigid} if the pair $\left(x_r,\tau^{-p}(x_1)\right)$ is left-weighted; roughly speaking, this means that the left normal form written cyclically is left-weighted everywhere. 
Also, we recall that to each \emph{simple element} $s$ of~$G$ (that is, $s$~is a positive left and right divisor of~$\Delta$), is associated its \emph{right complement}: $\partial s=s^{-1}\Delta$, which is also a simple element. We extend this notion of right complement to each element $y$ of~$G$ with infimum~0: $\partial y=y^{-1}\Delta^{\sup(y)}$. In terms of the left normal form, if $y=y_1\ldots y_r$, then the normal form of $\partial y$ is $y'_r\ldots y'_1$ where $y'_i=\tau^{r-i}(\partial y_i)$, for $i=1,\ldots, r$.

The following formulae will be helpful and are well-known, see~\cite{Elrifai-Morton}.
For any $x,y\in G$, $p\in \mathbb Z$, 
$$\inf(\Delta^p x)=p+\inf(x),\ \ \sup(\Delta^px)=p+\sup(x).$$ 
$$\inf(xy)\geqslant \inf(x)+\inf(y),\ \ \ \sup(xy)\leqslant \sup(x)+\sup(y).$$
$$\inf(y^{-1})=-\sup(y),\ \ \ \sup(y^{-1})=-\inf(y).$$

\begin{definition}\label{D:Absorbable}
We say that an element $y$ of $G$ is \emph{absorbable} if two conditions are satisfied:
\begin{itemize}
\item $\inf(y)=0$ or $\sup(y)=0$,
\item there exists some $x\in G$ such that $$\begin{cases}
\inf(xy)=\inf(x)\ \text{ \ and} \\
 \sup(xy)=\sup(x).
\end{cases}$$
\end{itemize}
In this case we also say more precisely that $y$ is absorbable by $x$ or that $x$ \emph{absorbs} $y$. 
\end{definition} 

\begin{remark}
Definition~\ref{D:Absorbable} is very practical for our purposes, but it might not be the most suitable one for generalizing our techniques to other frameworks. We suggest another possible definition: say an element $y$ of $G$ is \emph{absorbable}$'$ if there exists an $x\in G$ 
such that for every initial segment $y^{(i)}=y_1\ldots y_i$ of the mixed normal form $y=y_1\ldots y_l$ we have: 
$$\inf(xy^{(i)})=\inf(x)\ \text{ \ and \ } \sup(xy^{(i)})=\sup(x).$$
(Note that we dropped the requirement that $\inf(y)=0$ or $\sup(y)=0$.)
This alternative definition is not quite equivalent to Definition~\ref{D:Absorbable}, but almost: every absorbable element is also absorbable$'$, and conversely, every absorbable$'$ element is the product of at most two absorbable elements, namely the positive and the negative parts of its mixed normal form.
\end{remark}

The following are immediate consequences of Definition~\ref{D:Absorbable}: 

\begin{lemma}\label{L:BasicAbsorb}
Let $y$ be an element of $G$. 
\begin{itemize}
\item[(i)] If $y$ is absorbable then there exist $k\in \mathbb N$ and $y_1,\ldots, y_k$ simple elements so that the left normal form of $y$ is $y_1\ldots y_k$ or $\Delta^{-k}y_1\ldots y_k$. 
\item[(ii)] $y$ is absorbable if and only $y^{-1}$ is absorbable. This is also equivalent to $\tau(y)$ and $\tau(y^{-1})$ being absorbable.
\end{itemize}
\end{lemma}

\begin{proof}
(i) This is just a rewriting of the condition that $\inf(y)=0$ or $\sup(y)=0$ from Definition~\ref{D:Absorbable}.

(ii) Because $\inf(y^{-1})=-\sup(y)$ and $\sup(y^{-1})=-\inf(y)$, the first condition for absorbability is  satisfied by both $y$ and $y^{-1}$ or none.   
Moreover, if $y$ is absorbable by $x$, then $\inf(xy)=\inf(x)=\inf((xy)y^{-1})$ and $\sup(xy)=\sup(x)=\sup((xy)y^{-1})$. This shows that $y^{-1}$ is also absorbable, by $xy$. For later reference, we make the additional observation that $x$ and $xy$, which  absorb $y$ and $y^{-1}$, respectively, have the same sup and the same inf.

For the rest of statement (ii), just note that $y$ is absorbable by $x$ if and only if $\tau(y)$ is absorbable by $\tau(x)$.
\end{proof}

We shall see in Example~\ref{I:sInvDelta}(5) that the complement~$\partial y$ of an absorbable element~$y$ is not necessarily absorbable.

The following observation indicates that being absorbable may be a fairly rare property. 
\begin{lemma}\label{L:subword}
Any positive subword of a positive absorbable element is absorbable.
That is, suppose that a positive absorbable element~$y$ of~$G$ can be written as a product of three positive elements $y=uvw$ (with possibly $u=1$ or $w=1$). Then $v$ is absorbable.
\end{lemma} 

\begin{proof}
First notice that $\inf(v)=0$. Let $x$ be such that $\inf(xy)=\inf(x)$ and $\sup(xy)=\sup(x)$. Then we claim that $\inf((xu)v)=\inf(xu)$ and $\sup((xu)v)=\sup(xu)$, implying that $v$ is absorbable. 
In order to prove the claim, we recall the inequalities $\inf(a)\leqslant \inf(ab)$ and $\sup(a)\leqslant \sup(ab)$ for any $a,b\in G$ with $b$ positive. They imply
$$\inf(x)\leqslant \inf(xu)\leqslant \inf(xuv)\leqslant \inf(xuvw)=\inf(x),$$ 
$$\sup(x)\leqslant \sup(xu)\leqslant \sup(xuv)\leqslant \sup(xuvw)=\sup(x).$$
\end{proof}

\begin{lemma}\label{L:absorblength}
Let $y$ be an absorbable element with canonical length $k$. Then there exists $x$ with infimum 0 and supremum $k$ which absorbs $y$.
Moreover $k$ is the smallest possible number of factors in an element with infimum 0 absorbing $y$.  
\end{lemma}

Before giving the proof, we mention that Lemma~\ref{L:absorblength} yields, in principle, an algorithm for testing whether any given element~$y$ of $G$ is absorbable. It suffices to test, for every~$x\in G$ with $\inf(x)=0$ and $\sup(x)=\ell(y)$, whether $x$ absorbs~$y$. We do not know if there exists a polynomial-time algorithm for testing absorbability. 

\begin{proof}[Proof of Lemma~\ref{L:absorblength}]
If $k=0$ there is nothing to prove. 
Let $y$ be absorbable by $\hat x=\Delta^px$, with $\inf(x)=0$; then $\inf(xy)=\inf(\hat xy)-p=\inf(\hat x)-p=\inf(x)$ and similarly for the supremum, showing that $y$ is absorbable by $x$. 

Thus $y$ is absorbed by an element $x$ with $\inf(x)=0$. We have to show that we can take $x$ with the same length~$k$ as~$y$, and that this $k$ is minimal. We can restrict our attention to the case where $y$ is positive, i.e.\  $\inf(y)=0$; this is because $y$ and $y^{-1}$ can be absorbed by elements of the same length, as seen in the proof of Lemma~\ref{L:BasicAbsorb}.

From now on we assume that $y$ is positive and $y=y_1\ldots y_k$ is its left normal form. The absorbing element $x$ with $\inf(x)=0$ is at least of length~$k$, because $\sup(x)=\sup(xy)\geqslant \sup(y)=k$. We have to prove the existence of such an~$x$ with length exactly~$k$. More precisely, if $x=x_1\ldots x_l \, x_{l+1}\ldots x_{l+k}$ absorbs~$y$, we will show that so does $\tilde x=x_{l+1}\ldots x_{l+k}$.
First, by hypothesis $\inf(xy)=\inf(x)=0$, so 
$$0= \inf(\tilde x)\leqslant\inf(\tilde x y)\leqslant \inf(x y)=\inf(x)=0$$
and the condition on the infima is satisfied.
It remains to be shown that for all $i=1,\ldots, k$ the left normal form of $\tilde x y_1\ldots y_i$ has only $k$ letters. 

This is based on the following two observations. Firstly, if $z_1\ldots z_r$ is a left normal form with $z_1\neq  \Delta$ and if $s$ is a simple element with $\inf (zs)=0$, then the left normal form of $zs$ also has $r$ letters if and only if $z_rs$ is simple. 
Otherwise this left normal form has $r+1$ letters. Moroever, in the former case, if $r\geqslant 2$, for $j=2,\ldots, r$, the $j$th letter of the left normal form of $zs$ is fully determined by $s$ and $z_{j-1},\ldots, z_r$ (all the preceding letters of~$z$do not enter into consideration); this is our second observation. These two facts follow by inspection of the procedure for calculating normal forms explained in~\cite{Gebhardt-GM}, Proposition~1. 

Since $\sup(xy_1)=\sup(x)$, the first observation tells us that $x_{l+k}y_1$ must be simple, which in turn implies that $\sup(\tilde xy_1)=\sup(\tilde x)$. 
This terminates the proof if $k=1$. Moreover, if $k\geqslant 2$, the second observation implies that for $j= 2,\ldots, k$, the $j$th letter of the left normal form of $\tilde x y_1$ coincides with the $l+j$th letter of the left normal form of $x y_1$. 
Applying again the first observation together with the absorbability of $y$ in $x$, hence of $y_2$ in $xy_1$, we see that $\sup(\tilde x y_1y_2)=k$ and we are done if $k=2$. Moreover (if $k\geqslant 3$), for $j=3,\ldots, k$, thanks to the second observation, the $j$th letter of the normal form of $\tilde x y_1 y_2$ coincides with the $l+j$th letter of the normal form of $x y_1 y_2$. Continuing inductively, we obtain the desired result that $\sup(\tilde x y_1\ldots y_i)=k$, for all $i=1,\ldots, k$.
\end{proof}

\begin{example}
\begin{itemize}
\item[(1)] Whenever $n\geqslant 3$, in the ``classical" Garside structure on the free abelian group $(\mathbb Z^n,\mathbb N^n,(1,1,\ldots,1))$, any multiple of a standard generator is absorbable.   
\item[(2)] In the braid group $B_4$ with its classical Garside structure, the braid $y=\sigma_1^2\sigma_2^2\sigma_3^2\sigma_2^2\sigma_1$ is absorbable, e.g.\ by $x=\sigma_1\sigma_2^4\sigma_1^2\sigma_2\sigma_3$: we calculate 
$$\sigma_1\sigma_2\ .\ \sigma_2\ .\ \sigma_2\ .\ \sigma_2\sigma_1\ .\ \sigma_1\sigma_2\sigma_3 \ \cdot \ \ \sigma_1\ .\ \sigma_1\sigma_2\ .\ \sigma_2\sigma_3\ .\ \sigma_3\sigma_2\ .\ \sigma_2\sigma_1 =\phantom{OOOOOOO}$$ 
$$\phantom{OOOOOOOO}= \sigma_1\sigma_2\sigma_1\ .\ \sigma_1\sigma_2\sigma_1\sigma_3\ .\ \sigma_1\sigma_2\sigma_3\sigma_2\ .\ \sigma_2\sigma_3\sigma_2\ .\ \sigma_2\sigma_3\sigma_2\sigma_1$$
Notice that $y$ is pseudo-Anosov and rigid. This is the most surprising example of an absorbable braid we know, and the longest non-reducible one. 
\item[(3)] The length 2 braid $(\sigma_1\sigma_3)^2$ in~$B_4$ is not absorbable, as shows Lemma~\ref{L:absorblength} together with an inspection of all braids with infimum 0 and supremum 2. Nor is absorbable any 4-braid with infimum 0 and left normal form $x_1\ldots x_r$ such that for some $i=1,\ldots, r-1$, $x_i\succcurlyeq \sigma_1\sigma_3$ and $\sigma_1\sigma_3\preccurlyeq x_{i+1}$, by Lemma~\ref{L:subword}.   
\item[(4)]\label{I:sInvDelta} In any Garside group, if $s$ is an atom, then the simple element $y=s^{-1}\Delta$ is not absorbable. Indeed, if $y$~was absorbable then, by Lemma~\ref{L:absorblength}, it could be absorbed by a \emph{simple } element~$x\neq 1$. We would then have $xy \prec \Delta$. Since left divisors of $\Delta$ are also right divisors of~$\Delta$, this means that there exists a simple element $a\neq 1$ satisfying $axy=\Delta$. By combining this with the equality $sy=\Delta$, we obtain $ax=s$, contradicting the hypothesis that~$s$ is an atom.
\item[(5)] As an application of the previous example, in the braid group $B_n$ with its classical Garside structure, the braid $\sigma_i^{-1}\Delta$, for any $i$ between 1 and $n-1$, is not absorbable (even though it is the complement of the absorbable braid $\sigma_i$). 
\end{itemize}
\end{example}


\subsection{The additional length complex}

\begin{definition}\label{D:AddLengthCx}
Suppose $G$ is a Garside group, the group of fractions of a Garside monoid $(P,\Delta)$. We define the \emph{additional length complex} $\mathcal C_{AL}(G,P,\Delta)$ (generally abbreviated as $\mathcal C_{AL}(G)$, or even $\mathcal C_{AL}$) to be the following (usually locally infinite) connected graph. 
\begin{itemize}
\item The vertices are in correspondence with $G/\langle \Delta\rangle$, 
that is the cosets $g\Delta^{\mathbb Z}=\{g\Delta^z \ | \ z\in\mathbb Z\}$. 
For each vertex $v$ we have a unique distinguished representative with infimum 0, which we denote~$\underline v$. 
\item Two vertices $v=\underline{v}\Delta^{\Z}$ and $w=\underline{w}\Delta^{\Z}$ of $\mathcal C_{AL}$ are connected by an edge if one of the following happens:
\begin{enumerate} 
\item There exists a non-trivial, non-$\Delta$ simple element $m$ so that the element $\underline{v} m$ represents the coset $w$. 
This is equivalent to saying that there is a simple element $m'\neq 1,\Delta$ such that $\underline{w}m'$ belongs to the coset $v$.
(This first type of edges is as in Bestvina's normal form complex, see~\cite{CMW}.) 

\item There exists an absorbable element $y$ of~$G$ so that $\underline{v}y$ belongs to the coset~$w$. 
This is equivalent to saying that there is an absorbable element $y'$ of $G$ so that $\underline{w}y'$ belongs to the coset $v$. 
\end{enumerate}
\end{itemize}

As usually, a metric structure on the above complex is given simply by declaring that every edge is of length 1. We call this metric the \emph{additional length metric}. The distance between two vertices $v$ and $w$ in $\mathcal C_{AL}$ will be denoted $d_{AL}(v,w)$. The group~$G$ acts on the left by isometries on this complex.
\end{definition}

\begin{remark}
(a) The idea of this definition is that in the additional length complex, a group element $y$ is close to the identity if ``multiplying by the element~$y$ does not necessarily add any length'' -- hence the name of the complex.

(b) If, in Definition~\ref{D:AddLengthCx}, we leave out the second type of edges, then we obtain precisely the $1$-skeleton of the Bestvina normal form complex as described in~\cite{CMW}. Thus the additional length complex can be thought of as a squashing of the Bestvina normal form complex.
\end{remark}

Next we shall associate to each pair of vertices $v,w$ of $\mathcal C_{AL}$ a preferred path $A(v,w)$ between $v$ and $w$:
\begin{definition} (See Definition 6.1. in~\cite{CMW}). 
Let $v=\underline{v}\Delta^{\Z}$ and $w$ be two vertices of~$\mathcal C_{AL}$. 
\begin{itemize}
\item The \emph{preferred path} $A(1,v)$ is
the connected subgraph of $\mathcal C_{AL}$ given by the left normal form of $\underline v$. That is, if $v_1\ldots v_{\sup(\underline v)}$ is the left normal form of $\underline v$, $A(1,v)$ is the path starting at $1$ whose edges are successively labeled $v_1,\ldots,v_{\sup(\underline v)}$; for $i=0,\ldots, \sup(\underline v)$, the distinguished representative of the $i$th vertex along $A(1,v)$ is $\Delta^{i}\wedge \underline v$.
\item The preferred path $A(v,w)$ from $v$ to $w$ is given by the translation on the left by $\underline{v}$ of the preferred path 
$A(1,(\underline{v}^{-1}\underline{w})\Delta^{\mathbb Z})$. That is, if $x=x_1\ldots x_r$ is the left normal form of the distinguished representative of $({\underline v}^{-1}\underline w)\Delta^{\Z}$, $A(v,w)$ is the path of length $r$ starting at $v$ whose edges are successively labeled $x_1,\ldots, x_r$.  
\end{itemize}
\end{definition}
Note that the path $A(v,w)$ uses only the edges of $\mathcal C_{AL}$ coming from the Cayley graph of~$G$ (with respect to the divisors of $\Delta$), not those coming from absorbable elements, and that the length of the path may well be much larger than the distance between $1$ and $v$. As normal forms are unique, if two vertices $v$ and $w$ are connected by a path $\gamma$ whose edges are labeled by simple elements $s_1,\ldots, s_r$ satisfying that for $i=1,\ldots, r-1$, $(s_i,s_{i+1})$ is a left-weighted pair, then $\gamma=A(v,w)$.

In order to get a more detailed picture of this family of paths, we claim the following:

\begin{lemma}\label{L:pgcd}
Let $v=\underline{v}\Delta^{\mathbb Z}$ and $w=\underline{w}\Delta^{\Z}$ be two vertices of $\mathcal C_{AL}$. Then 
$A(v,w)$ is the concatenation of the paths $A(v,(\underline{v}\wedge\underline{w})\Delta^{\mathbb Z})$ and $A((\underline{v}\wedge\underline{w})\Delta^{\mathbb Z},w)$. 
I.e., the preferred path between $v$ and $w$ passes through the vertex $(\underline{v}\wedge \underline{w})\Delta^{\Z}$
\end{lemma}

\begin{proof}
Set $d=\underline v \wedge \underline w$. We have positive elements $a$ and $b$ such that $\underline v=da$, $\underline w=db$ and $a\wedge b=1$. By definition $A(v,w)$ is the left translate by $\underline v$ of the path $A(1,(\underline v^{-1}\underline w)\Delta^{\mathbb Z})$, which connects the identity vertex with the vertex represented by $\underline v^{-1}\underline w$. We shall see that the distinguished representative of the latter vertex is the element $\partial a\cdot \tau^r(b)$, where $r$ is the supremum of $a$.  Indeed, we have 
$$\partial a\cdot\tau^r(b)=a^{-1}\Delta^r\cdot\tau^r(b)=a^{-1}b\Delta^r=(a^{-1}d^{-1})(db)\Delta^r=\underline v^{-1}\underline w\Delta^r,$$ 
which shows that our element represents the correct vertex. Moreover, if we write the left normal forms as $a=a_1\ldots a_r$ and $b=b_1\ldots b_s$, we have
$$\partial a\cdot\tau^r(b)= \partial a_r\ldots \tau^{r-1}(\partial a_1)\cdot\tau^r(b_1)\ldots \tau^r(b_s),$$
which is in left normal form as written because, as $a\wedge b=1$, 
$$(\tau^{r-1}(\partial a_1),\tau^r(b_1))$$ 
is a left-weighted pair.
Thus $\inf(\partial a\cdot\tau^r(b))=0$ and this shows that $\partial a\cdot\tau^{r}(b)$ is the desired distinguished representative. This says moreover that the path $A(1,\underline v^{-1}\underline w\Delta^{\mathbb Z})$ is the concatenation of the paths $A(1,\partial a\Delta^{\Z})$ and $A(\partial a\Delta^{\mathbb Z},\partial a\tau^r(b)\Delta^{\Z})$, that is, of $A(1,a^{-1}\Delta^{\Z})$ and $A(a^{-1}\Delta^{\Z},{\underline{v}}^{-1}\underline w \Delta^{\Z})$. 
After translation by $\underline v$, using the equality $\underline va^{-1}=d$, we see that our path $A(v,w)$ is the concatenation of $A(v,d\Delta^{\Z})$ and $A(d\Delta^{\Z},w)$, as we wanted to show. 
\end{proof}

\begin{lemma}\label{L:Symmetry}
The preferred paths are symmetric: for any vertices $v,w$ of $\mathcal C_{AL}$, we have $A(v,w)=A(w,v)$.
\end{lemma}

First, note that the lemma has nothing to do with our strange metric, the analogue result  is also true in Bestvina's normal form complex -- see Lemma 6.4 in~\cite{CMW}).

\begin{proof}
As in the proof of Lemma~\ref{L:pgcd}, set $d=\underline v\wedge \underline w$. We have two elements $a,b$ of $G$, with $\inf(a)=\inf(b)=0$, $\underline v=da$, $\underline w=db$ and $a\wedge b=1$. Set moreover $r=\sup(a)$ and $s=\sup(b)$.
By definition, $A(v,w)$ is the left translate by $\underline v$ of the path $A(1,(\underline v^{-1}\underline w)\Delta^{\Z})$ and we have seen in the proof of Lemma~\ref{L:pgcd} that the latter is given by the left normal form of $\partial a\cdot \tau^r(b)$. 
Similarly, $A(w,v)$ is the left translate by $\underline w$ of the normal form of $\partial b\cdot \tau^s(a)$. 
First we note that both paths have the same length, namely $r+s$. For $0\leqslant i\leqslant r+s$, we will show that
$\underline v(\Delta^i\wedge \partial a\cdot\tau^r(b))$ represents the same vertex as $\underline w(\Delta^{r+s-i}\wedge \partial b\cdot\tau^s(a))$, hence showing the lemma. In other words, when traveling along the path $A(v,w)$ or along the path $A(w,v)$, one meets exactly the same vertices of $\mathcal C_{AL}$, but in the reverse order.

First, assume $0\leqslant i<r$. On the one hand, 
$$\underline v(\Delta^i\wedge \partial a\cdot\tau^r(b))=da(\Delta^i\wedge \partial a)=da_1\ldots a_{r-i}\Delta^i.$$
On the other, 
$$\underline w(\Delta^{r+s-i}\wedge \partial b\cdot\tau^s(a))=db\partial b(\Delta^{r-i}\wedge \tau^s(a))=da_1\ldots a_{r-i}\Delta^s.$$

Next, assume that $r< i\leqslant r+s$, that is $i=r+j$, for $0< j\leqslant s$. 
On the one hand, 
$$\underline v(\Delta^{r+j}\wedge \partial a\cdot\tau^r(b))=da\partial a(\Delta^j\wedge \tau^r(b))=d b_1\ldots b_j\Delta^r.$$
On the other, 
$$\underline w(\Delta^{r+s-(r+j)}\wedge \partial b\cdot\tau^s(a))=db(\Delta^{s-j}\wedge \partial b)=db_1\ldots b_j\Delta^{s-j}.$$ 

Finally, if $i=r$, we have $\underline v(\Delta^r\wedge \partial a\cdot \tau^r(b))=da\partial a=d\Delta^r$ and 
$\underline w(\Delta^{s}\wedge \partial b \cdot \tau^s(a))=db\partial b=d\Delta^s$.
\end{proof}

Here is our main result

\begin{theorem}\label{T:main}
For any Garside group $(G,P,\Delta)$, the complex $\mathcal C_{AL}$ is 60-hyperbolic. Moreover, the family of paths $A(v,w)$ with $v,w\in G/\langle \Delta\rangle$, forms a family of uniform unparametrized quasi-geodesics in the complex: for any $v,w\in G/\langle \Delta\rangle$, the Hausdorff distance between $A(v,w)$ and a geodesic from $v$ to $w$ is bounded above by~39.
\end{theorem}

\begin{remark}
Note that the hyperbolicity-constant is bounded independently of~$(G,P,\Delta)$.
\end{remark}

\begin{proof}[Proof of Theorem \ref{T:main}]
First recall Proposition 3.1 in~\cite{Bowditch} (the ``guessing geodesics lemma''):

\begin{proposition}\label{P:GuessingGeodesics}
Given $h\geqslant 0$, there is some $k\geqslant 0$ with the following property. Suppose that $X$ is a connected graph and that for each pair of vertices $x,y$ of $X$, we have associated a connected subgraph $A(x,y)\subseteq X$, with $x,y \in A(x,y)$. Suppose that 
\begin{itemize}
\item For all vertices $x,y$ of $X$ connected by an edge, $A(x,y)$ has diameter in $X$ at most $h$. 
\item For all vertices $x,y,z$ of $X$, $A(x,y)$ is contained in an $h$-neighborhood of the union $A(x,z)\cup A(y,z)$.
\end{itemize}
Then $X$ is $k$-hyperbolic. Moreover, if $m$ is any positive real number so that $2h(6+\log_2(m+2))\leqslant m$, we can take any number $k\geqslant \frac{3}{2}m-5h$. Moreover, for all vertices $x,y$ of $X$, the Hausdorff distance between $A(x,y)$ and any geodesic between $x$ and $y$ is bounded above by $m-4h$.
\end{proposition}

We will show that the hypotheses of Proposition~\ref{P:GuessingGeodesics} are satisfied with $X=\mathcal C_{AL}$, and $h=2$. Then the inequality $2\cdot 2\cdot (6+\log_2(m+2))\leqslant m$ holds for the positive number $m=$46,{\small 5}. This yields the estimate $k=60$ and the statement about the unparametrized quasi-geodesic paths.

First we look at the first condition: preferred paths between adjacent vertices have uniformly bounded diameter in $\mathcal C_{AL}$.
\begin{lemma}\label{L:PreferredPathsNoLoops}
Let $v,w$ be two vertices of $\mathcal C_{AL}$ such that $d_{AL}(v,w)=1$. Then the diameter in $\mathcal C_{AL}$ of $A(v,w)$ is equal to~1. 
\end{lemma}

\begin{proof}
We may assume that $v=1$. 
If $\sup(\underline w)=1$, then there is nothing to prove: $A(1,w)$ just consists of an edge with two vertices. 
Otherwise, $\sup(\underline w)>1$. As there is an edge between 1 and $w$, there exists an absorbable element $y$ so that 
$y=\underline w\Delta^{k}$, for some $k\in \Z$. By definition of absorbable elements, 
this implies either $k=0$ (in which case $y=\underline w$ is absorbable and positive), or $k=-\sup(\underline w)$. 
In the first case, $A(1,w)$ is given by the left normal form of $y=\underline w$; by Lemma~\ref{L:subword}, this has diameter 1 in $\mathcal C_{AL}$. In the second case, we look at the path $A(w,1)$ which is the translate by $\underline w$ of the path 
$A(1,\underline w^{-1}\Delta^{\Z})$. The latter corresponds to the left normal form of the element $\partial \underline w$. 
But~$y$, and thus $y^{-1}$, are absorbable; and the equality $\partial \underline w=\tau^{-k}(y^{-1})$ shows that $\partial \underline w$ is also absorbable. Therefore, again by Lemma~\ref{L:subword}, the path $A(1,\underline w^{-1}\Delta^{\Z})$ has diameter~1 in $\mathcal C_{AL}$ as we needed to show.
\end{proof}

We now proceed to show the second condition: the 2-thinness of any triangle whose edges are our preferred paths. 

\begin{lemma}\label{L:2-thinness}
Let $u,v,w$ be three vertices of $\mathcal C_{AL}$. The triangle in $\mathcal C_{AL}$ with vertices $u$, $v$ and $w$, and with edges $A(u,v)$, $A(v,w)$ and $A(u,w)$ is $2$-thin: each edge is at Hausdorff distance at most 2 from the union of the other two edges.
\end{lemma}

\begin{proof}[Proof of Lemma~\ref{L:2-thinness}]
For the proof, first notice that without loss of generality we can assume that $u=1$. We then set, as in the above proofs, $d=\underline v\wedge \underline w$. We consider the elements $a,b$ of $G$ satisfying $\underline v=da$, $\underline w=db$ and $a\wedge b=1$. 
We also set $k=\sup(\underline v)$, $l=\sup(\underline w)$, $r=\sup(a)$, $s=\sup(b)$ and $p=\sup(d)$. 

\begin{figure}[htb]
\begin{center}\includegraphics[width=13cm]{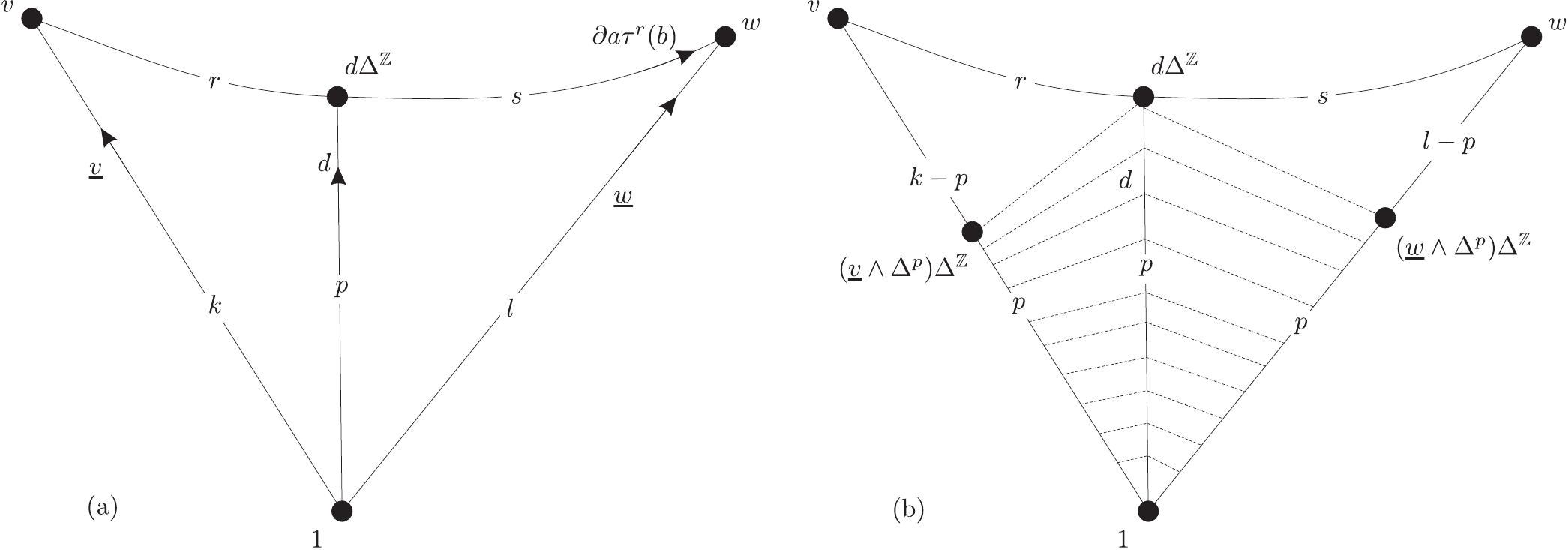}
\end{center}
\caption{(a) A triangle with vertices $1$, $v$ and $w$ and normal form edges. \ (b) How the triangle is squashed in $\mathcal C_{AL}$.}
\label{F:2-thinness}
\end{figure}

\begin{lemma}\label{L:InitialSegments}
The initial segments of length $p$ of $A(1,v)$ and $A(1,w)$ are at Hausdorff distance at most 2 in $\mathcal C_{AL}$.
\end{lemma}
\begin{proof}
First recall that for any integer $i=1,\ldots,k$,
the $i$th step on the preferred path $A(1,v)$ is at the vertex $(\underline v\wedge \Delta^{i})\Delta^{\Z}$, whose distinguished representative is exactly $\underline v\wedge \Delta^i$. Notice that $d$ itself is the distinguished representative of $d\Delta^{\mathbb Z}$.

It is sufficient to prove that the initial segment of $A(1,v)$ of length $p$ is at Hausdorff distance at most 1 from $A(1,d\Delta^{\Z})$ in $\mathcal C_{AL}$. 
Specifically, we claim that the respective $i$th steps of $A(1,v)$ and of $A(1,d\Delta^{\Z})$ are at distance at most 1 for any $i=1,\ldots p$, that is 
$$d_{AL}((\underline v \wedge \Delta^i)\Delta^{\mathbb Z},(d\wedge \Delta^i)\Delta^{\mathbb Z})\leqslant1.$$
But now observe that $d\wedge \Delta^i\preccurlyeq \underline v\wedge \Delta^i$, so that we can find a positive element~$y$ such that $(d\wedge \Delta^i)y=\underline v \wedge \Delta^i$. This element $y$ is absorbable by $d\wedge \Delta^i$ as $\sup(d\wedge \Delta^i)=\sup(\underline v \wedge \Delta^i)=i$ and $\inf(d\wedge \Delta^i)=\inf(\underline v \wedge \Delta^i)=0$. This shows the claim. 
\end{proof}

Lemma~\ref{L:InitialSegments} says that in our triangle, the two edges emanating from any vertex have distinguished initial segments (possibly consisting of a single vertex) which stay at distance at most 2 from each other; moreover, the respective end points of these initial segments are at distance at most 1 from a common vertex on the third edge. 

We shall now see that for each edge of our triangle, the respective distinguished initial segments emanating from its two extremities actually overlap (or at least share a common vertex on the given edge). This is a consequence of the following lemma. 

\begin{lemma}\label{L:Overlap}
We have $\sup(\partial \underline v\wedge \partial a\cdot \tau^r(b))\geqslant r$.
\end{lemma}

\begin{proof}
It suffices to exhibit a common prefix of $\partial \underline v$ and $\partial a$ of length $r$.
Our candidate is $U$, which we define to be the product of the $r$ first factors in the \emph{right} normal form of $\partial \underline v$. 
In other words, we have $U=\partial(\Delta^r\wedge^{\hspace{-0.7mm}\Lsh\hspace{0.7mm}} \underline v)$ (where $\wedge^{\hspace{-0.7mm}\Lsh\hspace{0.7mm}}$ denotes the right gcd in~$G$). 
It is by construction a prefix of $\partial \underline v$ of length $r$. It remains to be shown that it is also a prefix of $\partial a$. 
But notice that $a$, as a suffix of $\underline v$ of length r, is certainly a suffix of $\Delta^r\wedge^{\hspace{-0.7mm}\Lsh\hspace{0.7mm}} \underline v$, so that we can find a positive $R$ satisfying $\Delta^r\wedge^{\hspace{-0.7mm}\Lsh\hspace{0.7mm}} \underline v=Ra$. But now, $$\partial a=a^{-1}\Delta^r=(\Delta^r\wedge^{\hspace{-0.7mm}\Lsh\hspace{0.7mm}} \underline v)^{-1}R\Delta^r=U\tau^r(R).$$ This shows that $U$ is also a prefix of $\partial a$.
\end{proof}

Along the edge $A(1,v)$, we have on the one hand a distinguished initial segment emanating from 1 which has length $p$. 
On the other hand, Lemma \ref{L:Overlap} says that the distinguished initial segment emanating from $v$ has length at least $r$. Because $\underline v=da$, the length $k$ of the edge $A(1,v)$ is at most $p+r$.
Hence the two distinguished initial segments at least meet in a point along $A(1,v)$. 

This shows that any point of $A(1,v)$, and hence by symmetry any point on any edge of our triangle, is at distance at most 2 from some point in the union of the other two edges.
\end{proof}

Lemmas~\ref{L:PreferredPathsNoLoops} and~\ref{L:2-thinness} guarantee that the hypotheses of the Guessing Geodesics Lemma~\ref{P:GuessingGeodesics} are satisfied. This completes the proof of Theorem~\ref{T:main}. 
\end{proof}

\begin{openproblems}
\begin{enumerate}
\item What is the boundary at infinity of~$\mathcal C_{AL}(G)$?

\item One of the most powerful tools for studying mapping class groups are \emph{subsurface projections} in curve complexes \cite{MM2}. Is there a good analogue notion in $\mathcal C_{AL}$?

\item\label{Q:InfDiam} Under which conditions on~$G$ does $\mathcal C_{AL}(G)$ have infinite diameter? In Theorem~\ref{T:InfDiam} we shall prove that this is the case if $G$ is the braid group, equipped with the classical Garside structure; however, the condition that $G/Z(G)$ is infinite may actually be sufficient. The special case of Artin-Tits groups of spherical type deserves particular attention.

\item\label{Q:WPD} Does $G$ act acylindrically on~$\mathcal C_{AL}(G)$? Recall that mapping class groups act acylindrically on curve complexes~\cite{Bowditch08,PrzSisto}.

\item Is it true that ``generic'' elements of~$G$ act loxodromically on~$\mathcal C_{AL}$, and thus are analogue to pseudo-Anosov elements in mapping class groups?  
If the word ``generic'' is used in the sense of ``a random element in a large ball in the Cayley graph'', then the answer is positive in the special case of braid groups with the classical Garside structure -- see~\cite{CarusoWiestGeneric2,WAutomLoxGeneric}. The question is closely related to question~(\ref{Q:InfDiam}) above. 
If, by contrast, the word ``generic'' is used in the sense of ``the result of a long random walk in the Cayley graph'', then a positive answer would essentially be implied by a positive answer to question~(\ref{Q:WPD}) above, using~\cite{SistoGeneric}.

\item Consider the braid group $B_n$, equipped with its classical Garside structure. Is it true that $\mathcal C_{AL}(B_n)$ is quasi-isometric to $\mathcal{CC}(D_n)$, the curve complex of the $n$-times punctured disk (see Section~\ref{SS:QiWithCC})? (This conjecture is the reason why we think of $\mathcal C_{AL}$ as an analogue of the curve complex.)

\item If $G$ is a Garside group with two different Garside structures $(G,P,\Delta)$ and $(G,Q,\delta)$, are the additional length complexes $\mathcal C_{AL}(G,P,\Delta)$ and $\mathcal C_{AL}(G,Q,\delta)$ quasi-isometric? In particular, are the additional length complexes associated with the classical (respectively, dual) Garside structure of the braid group $B_n$ quasi-isometric? We conjecture that they are, since both should be quasi-isometric to $\mathcal{CC}(D_n)$. 
\item Is the automorphism group of~$\mathcal C_{AL}(G)$ commensurable with~$G$? Recall that the automorphism group of the curve complex is commensurable with the mapping class group, by Ivanov's theorem~\cite{Ivanov}.

\item Is there a fast algorithm for finding parametrized quasi-geodesics, or even geodesics, between any two given points in~$\mathcal C_{AL}$? (Note that the Garside normal form yields a fast algorithm for constructing \emph{unparametrized} quasi-geodesics.) To start with, is there a fast algorithm for deciding absorbability?

\item Is the construction principle of $\mathcal{C}_{AL}$ useful in contexts other than Garside groups, for instance for general mapping class groups, or for $\mathrm{Out}(F_n)$?
\end{enumerate}
\end{openproblems}


\section{The special case of the braid groups}\label{S:B_n}

Throughout this section we consider the special case where $G=B_n$, the braid group on $n$~strands, equipped with the classical Garside structure. For an excellent introduction to this structure, see~\cite{Elrifai-Morton}.

\begin{example}\label{E:InfDiam}
\begin{enumerate}
\item For $n=2$, $B_2$ is the infinite cyclic group generated by $\Delta_2=\sigma_1$, so $B_2/\langle\Delta_2\rangle$ as well as the associated additional length complex are trivial.
\item For $n=3$, the only absorbable braids are $\sigma_1$, $\sigma_2$ and their respective inverses. 
Therefore the additional length complex in that special case is nothing but Bestvina's normal form complex, which has infinite diameter.
\end{enumerate}
\end{example}


\subsection{Periodic and reducible braids act elliptically}\label{SS:PeriodReduc}

\begin{example}\label{E:DecompositionDelta} In $B_n$ with $n\geqslant 4$, the braid $\Delta^{k}$ (with $k\in \mathbb Z-\{0\}$) is the product of three absorbable braids. Indeed, suppose $k\geqslant 1$ and let $A=\sigma_1^k$, $B=\sigma_3^k$, and $C=A^{-1}B^{-1}\Delta^k$. Then $\Delta^k=A\cdot B\cdot C$. Moreover, $A$ and~$B$ can absorb each other, and $C$ can be absorbed by~$A$. It follows from Lemma \ref{L:BasicAbsorb}(ii) that $A^{-1}$, $B^{-1}$ and $C^{-1}$ are absorbable. Thus $\Delta^{-k}=C^{-1}B^{-1}A^{-1}$ is the product of three absorbable braids, hence showing the claim for negative powers, too. 
\end{example}

Recall that the braid group acts, on the left, on the set of isotopy classes of simple closed curves in the $n$-times punctured disk. In what follows, we shall take these punctures to be lined up horizontally. Also, by a \emph{round} curve we shall mean the isotopy class of a geometric essential circle (i.e. enclosing more than 1 and less than $n$ punctures). 

\begin{lemma}\label{L:ReducibleSmallDistance}
Suppose that $n\geqslant 4$. Any $n$-braid which sends a round curve to a round curve is a product of at most nine absorbable braids. In particular, every reducible braid with round reduction curves is a product of at most nine absorbable braids.
\end{lemma}

\begin{proof} 
Let $y$ be a braid sending a round curve to a round curve. As in the computation of a left normal form, we can get rid of the possible negative factors in $y$ at the cost of at most three absorbable braids (see Example \ref{E:DecompositionDelta}). As powers of $\Delta$ send round curves to round curves we may suppose that $y$ is a positive braid sending the round curve $\mathcal C$ to a round curve.  

We recall~\cite{BGN,CalvezStandard,GonzalezMenesesRed} that in any braid~$y$ which sends a round curve~$\mathcal C$ to a round curve, pushing the curve~$\mathcal C$ along the braid gives rise to a ``tube'' that stays round all along the braid~$y$. Thus $y$ can be written as the product $y=y_{\rm int}\cdot y_{\rm tub}$ of an interior braid $y_{\rm int}$ and a tubular braid $y_{\rm tub}$: in the interior braid the tube just goes straight down and only the strands inside the tube can cross each other. By contrast, the tubular braid~$y_{\rm tub}$ looks just like~$y$, except that all crossings between pairs of strands living in the tube have been removed. 
Figure \ref{F:RedAbsorbable} shows an example in $B_5$. 

We shall show that each of $y_{\rm int}$ and $y_{\rm tub}$ can be written as a product of three absorbable braids. We start with some notation. Firstly, we denote by $\Delta_\mathcal C$ the simple braid in which two strands cross if and only if they both start at punctures enclosed by $\mathcal C$.  
Secondly, let $i$ be an integer such that punctures number $i$ and~$i+1$ are enclosed by~$\mathcal C$.

In order to prove the claim concerning $y_{\rm int}$, let us first
suppose that $\mathcal C$ encloses strictly less than $n-1$ punctures, thus at least two punctures are not enclosed by $\mathcal C$. If there is a $j$ such that the punctures $j$ and $j+1$ are not enclosed by $\mathcal C$, then $y_{\rm int}$ can be absorbed by 
an appropriate power of $\sigma_j$ (namely, $\sup(y_{\rm int})$). Otherwise, only the first and the $n$th puncture are not enclosed by $\mathcal C$; then 
there is an appropriate value of $p$ (namely, $\sup(y_{\rm int})$) so that $\prod_{\iota=1}^{p}\tau^{\iota}(\sigma_1\ldots \sigma_{n-1})$ absorbs $y_{\rm int}$. 

\begin{figure}[htb]
\begin{center}\includegraphics[width=9cm]{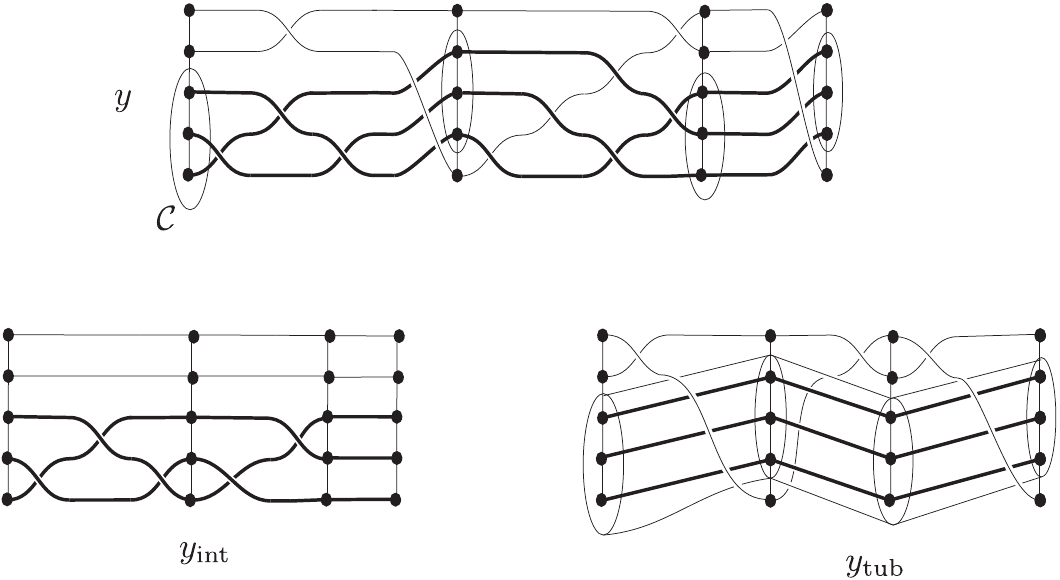} \end{center}
\caption{The braid $y=\sigma_1\sigma_2\sigma_1\sigma_4\sigma_3\sigma_2\sigma_1\cdot \sigma_1\sigma_2\sigma_1\sigma_3\sigma_2\sigma_4\cdot \sigma_4\sigma_3\sigma_2\sigma_2\sigma_1\in B_5$, the round curve $\mathcal C$ sent by $y$ to a round curve and the corresponding braids $y_{\rm int}=\sigma_1\sigma_2\sigma_1\cdot \sigma_1\sigma_2$ and $y_{\rm tub}=\sigma_4\sigma_3\sigma_2\sigma_1\cdot\sigma_1\sigma_2\sigma_3\sigma_4\cdot\sigma_4\sigma_3\sigma_2\sigma_1$; interior strands are depicted in bold lines. In this example, $y_{\rm int}$ is absorbable by $\sigma_4^2$. On the other hand, with $i=1$, $\sigma_i^3$ absorbs $y_{\rm tub}$.}
\label{F:RedAbsorbable}
\end{figure}

Suppose now that $\mathcal C$ encloses all the punctures but one. Up to conjugation by $\Delta$, which preserves absorbability (Lemma \ref{L:BasicAbsorb}(ii)), we may assume that the first puncture is not enclosed by~$\mathcal C$. We consider the decomposition $y_{\rm int}=\Delta_{\mathcal C}^k\cdot y'_{\rm int}$, where $k$ is a non-negative integer and $y'_{\rm int}$ is a positive braid not divisible by $\Delta_{\mathcal C}$. 
Then there is an appropriate value of $p$ (namely $\sup(y'_{\rm int})$) so that $y'_{\rm int}$ is absorbed by $\prod_{\iota=1}^{p} \tau^{p-\iota}(\sigma_{n-1}\ldots \sigma_1)$. 
The factor~$\Delta_{\mathcal C}^k$, on the other hand, can be further decomposed as $\Delta_{\mathcal C}^k=\sigma_i^k\cdot (\sigma_i^{-k}\Delta_{\mathcal C}^k)$. Both factors are absorbable by $\prod_{\iota=1}^{k} \tau^{k-\iota}(\sigma_{n-1}\ldots \sigma_{1})$. This completes the proof that $y_{\rm int}$ can be written as a product of three absorbable braids.

The proof for $y_{\rm tub}$ is similar: the braid $y_{\rm tub}$ can be decomposed into at most three factors which can all be absorbed by an appropriate power of~$\sigma_i$.\end{proof}

\begin{proposition}\label{P:PeriodReducActEllipt}
We consider the action of the braid group $B_n$, equipped with its classical Garside structure, on its additional length complex $\mathcal C_{AL}(B_n)$ by left multiplication. Then periodic and reducible elements act elliptically.
\end{proposition}
\begin{proof}
We recall that a braid is called periodic if it has some power which is also a power of $\Delta^2$. Since $\Delta^2$ acts trivially on the complex, periodic braids act as finite-order isometries on the complex: their action is thus elliptic. 
(Note that $\Delta$ does not act trivially: it sends any vertex $x\Delta^\mathbb Z$ to $\tau^{-1}(x)\Delta^\mathbb Z$.)

If a braid $x$ is reducible with a round reducing curve, then so is any of its powers. As seen in Lemma~\ref{L:ReducibleSmallDistance}, the orbit of the trivial braid under the action of $x$ remains at distance at most~$9$ from the trivial braid. This means that $x$ acts elliptically.

In order to deal with the case of braids which are reducible but without round reduction curves, we remark that such braids are conjugate to reducible braids with round reduction curves, and therefore they act elliptically, too.
\end{proof}


\subsection{$\mathcal C_{AL}(B_n)$ has infinite diameter}\label{SS:InfDiam}

\begin{theorem}\label{T:InfDiam}
Let $B_n$ be the braid group on $n$ strands ($n\geqslant 3$), equipped with the classical Garside structure. Then the complex $\mathcal C_{AL}(B_n)$ has infinite diameter. \end{theorem}

\begin{proof} For $n=3$, this is the statement in Example \ref{E:InfDiam} (2). Our strategy for proving Theorem \ref{T:InfDiam} is to actually construct elements whose action on the additional length complex is loxodromic. 
For every braid index $n\geqslant 4$, we will construct a special braid $x_n$ of infimum~0 such that the vertex $x_n^{N}\Delta^{\Z}$ ($N\in \mathbb N$) is at a distance at least $\frac{N}{2}$ from the identity vertex of~$\mathcal C_{AL}$. (As an aside, we conjecture that the action of $x_n$ on $\mathcal C_{AL}(B_n)$ is weakly properly discontinuous, \cite{BestvinaFujiwara02, PrzSisto}.)

We start with the construction of our special braids $x_n$. The rough idea is that $x_n$~should contain something like the ``blocking braids'' of~\cite{CarusoWiestGeneric2} in order to give~$x_n$ a very strong rigidity property (see Propositions~\ref{P:BetweenPowers} and \ref{P:InitialSegment}), but it should also contain pieces which prevent both $x_n$ and $\partial x_n$ from being absorbable. Here are the details.

Recall the \emph{shift} morphism $\text{sh}$ from $B_{\infty}$ to $B_{\infty}$ given by $\sigma_i\mapsto \sigma_{i+1}$ and the reverse antiautomorphism $\text{rev}$, $\sigma_{i_1}\sigma_{i_2}\ldots\sigma_{i_l} \mapsto \sigma_{i_l}\ldots \sigma_{i_2}\sigma_{i_1}$; 
for fixed $n$ we also note $\tau_n$ the conjugation by $\Delta_n$ inside $B_n$: $x\mapsto \Delta_n^{-1}x\Delta_n$.\\

Let us first define the 4-strand braid 
$$x_4=\sigma_2\cdot\sigma_2\sigma_1\sigma_3\cdot\sigma_1\sigma_3\sigma_2\sigma_1\sigma_3\cdot\sigma_1\sigma_3\sigma_2\sigma_1\sigma_3\cdot\sigma_1\sigma_3\sigma_2\cdot\sigma_2.$$

Then for $n\geqslant 5$, we define $u_n,x_n\in B_n$ as follows:

$$u_n=\text{sh}\left(\sigma_{\lfloor\frac{n-2}{2}\rfloor}^{-1}\Delta_{n-2}\right)\left(\sigma_{1}\ldots\sigma_{\lfloor\frac{n-1}{2}\rfloor}\right)\left(\sigma_{n-1}\ldots\sigma_{\lfloor\frac{n+3}{2}\rfloor}\right),$$

where $\lfloor \cdot\rfloor$ stands for the integer part, and 

\begin{multline*}
x_n=\text{sh}^{\lfloor\frac{n-3}{2}\rfloor}\left(\sigma_2\cdot\sigma_2\sigma_1\sigma_3\right)\cdot 
\prod^n_{\mathclap{\substack{k=5\\
                   k\equiv n \pmod 2\\
                   }}} \left( \text{sh}^{\lfloor\frac{n-k}{2}\rfloor}\left(\tau_{k}^{\lfloor\frac{k+1}{2}\rfloor}\left(u_{k}\right)\right)\right)
\cdot\\
\tau_n^{\lfloor\frac{n+1}{2}\rfloor}\left(\sigma_{\lfloor\frac{n+1}{2}\rfloor}^{-1}\Delta_n\right)\cdot \tau_n^{\lfloor\frac{n+1}{2}\rfloor}\left(\sigma_{\lfloor\frac{n}{2}\rfloor}^{-1}\Delta_n\right)\cdot\\
\text{rev}\left(\text{sh}^{\lfloor\frac{n-3}{2}\rfloor}\left(\sigma_2\cdot\sigma_2\sigma_1\sigma_3\right)\cdot \prod^n_{\mathclap{\substack{k=5\\
                   k\equiv n \pmod 2\\
                   }}}  \text{sh}^{\lfloor\frac{n-k}{2}\rfloor}\left(\tau_{k}^{\lfloor\frac{k+1}{2}\rfloor}\left(u_{k}\right)\right)\right).
                   \end{multline*}

The braids $x_9$ and $x_{10}$ are depicted in Figure~\ref{F:bloquante}. 
Note that for each $n\geqslant 4$, we have $\inf(x_n)=0$ and $\ell(x_n)=\sup(x_n)=2\cdot\lfloor \frac{n+1}{2}\rfloor+2$. 

\begin{figure}[ht]
\centerline{\includegraphics[height=11.65cm]{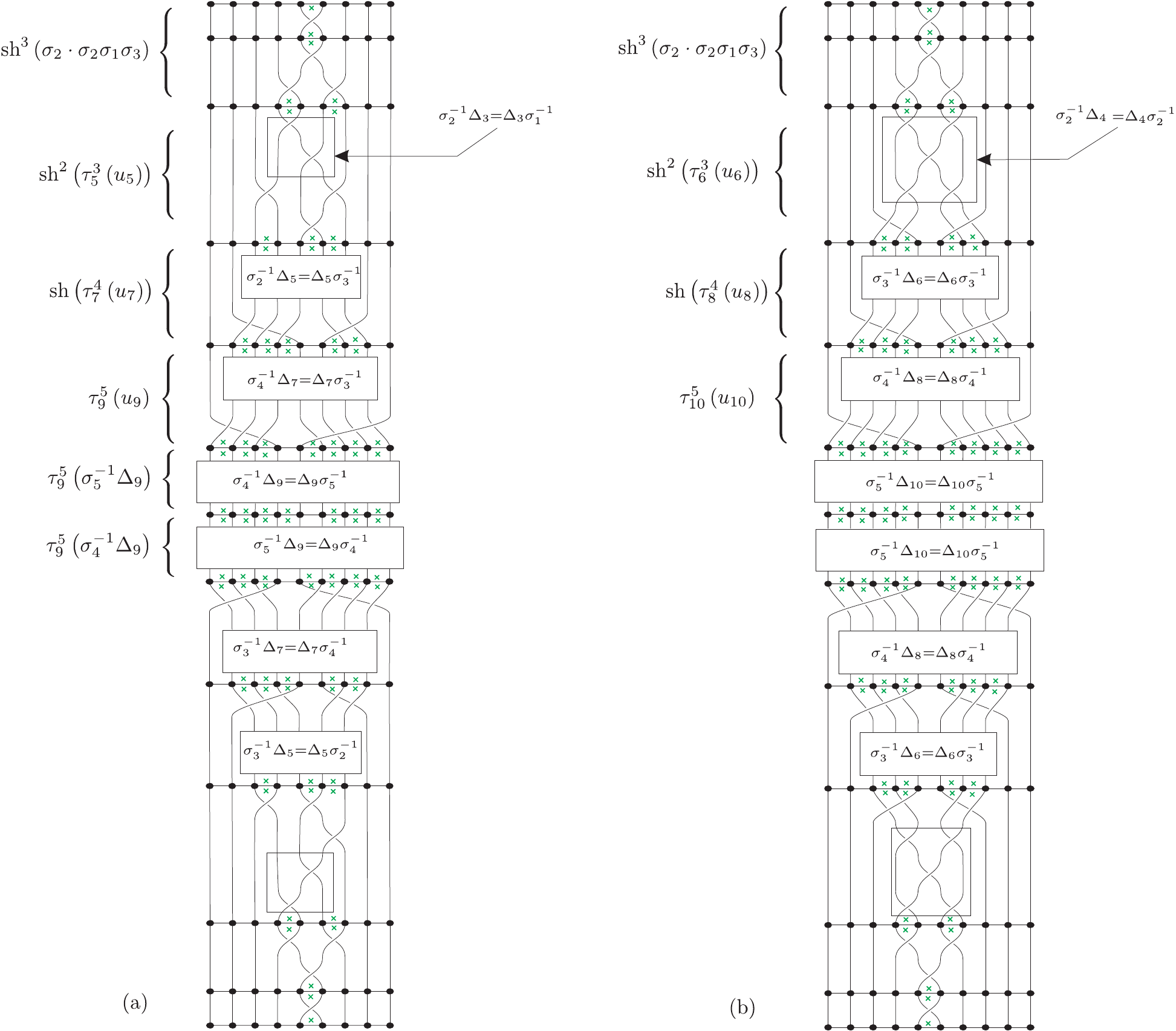}}
\caption{(a) The braid $x_{9}$. (b) The braid $x_{10}$. Both are shown in left (and right) normal form; at the beginning and the end of each of the factors, green crosses indicate which pair of adjacent strands cross in the considered factor.}
\label{F:bloquante}
\end{figure}

\begin{observation}\label{O:PropertiesOfx}
For each $n\geqslant 4$, the braid $x_n$ has the following properties:
\begin{enumerate}
\item The left and right normal forms of~$x_n$ are the same, 
\item The first factor $s_{\rm first}$ and the last factor $s_{\rm last}$ of the left and right normal form of~$x_n$ are the same and consist of a single atom: $s_{\rm first}=s_{\rm last}=\sigma_{\lfloor\frac{n+1}{2}\rfloor}$; in particular the pair $(s_{\rm last},s_{\rm first})$ is both right and left-weighted.
\item The left (and right) normal form of~$x_n$ contains a factor of the form $a^{-1}\Delta$, where $a$ is an atom. Specifically, $x_n$ contains the factor $\tau_n^{\lfloor\frac{n+1}{2}\rfloor}\left(\sigma_{\lfloor\frac{n+1}{2}\rfloor}^{-1}\Delta_n\right)$.

\end{enumerate}
\end{observation}

\begin{remark}
\begin{itemize}
\item Properties (1), (2), and (3) of $x_n$ above are the only ones used in the proof that $\mathcal C_{AL}(B_n)$ has infinite diameter.
\item Property (2) implies in particular that $x_n$ is rigid (see Section \ref{SS:Absorbable}).
\item $x_n$ is not absorbable, and neither is $\partial x_n$.  Indeed, Property (3), together with Lemma~\ref{L:subword}, implies that $x_n$ is not absorbable, as it contains a factor which by Example~\ref{I:sInvDelta} (5) is not absorbable. Similarly, the first and last factors of $x_n$ contain only one atom (Property (2)), so the corresponding factors of $\partial x_n$ are of the form  $\sigma_i^{-1}\Delta$; in particular, $\partial x_n$ is not absorbable, either.  
\item Another possible definition of $x_n$ would have been $x_n'=\mathrm{rev}(\alpha)\cdot \sigma_{n-1}^{-1}\Delta\cdot \Delta\sigma_{n-1}^{-1}\cdot \alpha$, where $\alpha$ is the ``blocking braid'' from~\cite{CarusoWiestGeneric2}. This braid~$x'_n$ is longer than the one presented above.
\end{itemize}
\end{remark}

From now on, we fix an arbitrary braid index $n\geqslant 4$ and we write $x=x_n$; we write $r$ for the number of factors in the normal form of $x$. Thus $r=2\cdot\lfloor \frac{n+1}{2}\rfloor+2$.

\begin{lemma}\label{L:PropertyOfx}
Suppose $v$ is a nontrivial positive suffix of~$x$, and $m$ a non-negative integer.

(a) The product $v\cdot x^m$ is in left normal form as written, i.e.\ the left normal form of the product is just the juxtaposition of the respective left normal forms of $v$ and of~$x^m$.

(b) For every nontrivial positive prefix $t$ of $vx^m$, we have $v\wedge t\neq 1$.  
\end{lemma}
\begin{proof}
As the last factor of the \emph{right} normal form of $x$ is $\sigma_{\lfloor\frac{n+1}{2}\rfloor}$ (Observation \ref{O:PropertiesOfx} (1-2)), the only simple suffix of $x$ is $\sigma_{\lfloor\frac{n+1}{2}\rfloor}$; in particular the last factor of the \emph{left} normal form of $v$ is $\sigma_{\lfloor\frac{n+1}{2}\rfloor}$. 
But this is also the first factor of the left normal form of $x^m$; because the pair $(\sigma_{\lfloor\frac{n+1}{2}\rfloor},\sigma_{\lfloor\frac{n+1}{2}\rfloor})$ is left-weighted, we obtain that the product $v\cdot x^m$ is in left normal form as written, proving~(a). In particular, $ \Delta\wedge vx^m=\Delta\wedge v$.

Now let $\sigma$ be a letter which divides $t$. In order to prove (b), it is sufficient to show that $\sigma\preccurlyeq v$. But we have: 
$\sigma\preccurlyeq \Delta\wedge t\preccurlyeq \Delta\wedge vx^m
=\Delta\wedge v\preccurlyeq v$. 
\end{proof}

The lemma allows to show that every prefix of some positive power of $x$ lies exactly between two successive powers of $x$ with respect to the prefix order. We first introduce some notation. 
\begin{notation}
For any braid $z$ with infimum 0 we define the non-negative integer $\lambda_x(z)=\max\{k\in \mathbb Z, \ x^k\preccurlyeq z\}$.
\end{notation} 

\begin{proposition}\label{P:BetweenPowers}
Let $z$ be a positive braid with infimum 0; let $\lambda=\lambda_x(z)$. 
Then the following are equivalent:
\begin{itemize}
\item[(a)] there exists a positive integer $m$ such that $z\preccurlyeq x^m$,
\item[(b)] $x^{\lambda}\preccurlyeq z\preccurlyeq x^{\lambda+1}$.
\end{itemize}
In this case, the product of the $\lambda r$ first factors of the left normal form of $z$ is exactly~$x^{\lambda}$, that is $\Delta^{\lambda r}\wedge z=x^{\lambda}$.
\end{proposition}

\begin{proof}
The direction (b) $\Longrightarrow$ (a) is obvious. To show the converse, we 
need to show that $z\preccurlyeq x^{\lambda+1}$. We may assume that $z$ is not a power of $x$.
Consider the braid $d=z\wedge x^{\lambda+1}$; by definition there exist positive braids $t$ and $v\neq 1$ such that 
$z=dt$, $x^{\lambda+1}=dv$ and $t\wedge v=1$. Note that $x^\lambda$ is a prefix of $d$, so that $v$ is a suffix of $x$. 
Now since $z=d t$ is a prefix of $x^m = d v x^{m-\lambda-1}$, we deduce that $t$ is a prefix of $v x^{m-\lambda-1}$. If $t$ is non-trivial, then we obtain by Lemma \ref{L:PropertyOfx} (b) that $t\wedge v\neq 1$, which is absurd. Thus $t=1$, which means that $z=z\wedge x^{\lambda+1}$, as we wanted to prove.

The second part of the statement follows from the calculation: 
$$x^\lambda=x^\lambda\wedge \Delta^{\lambda r}\preccurlyeq z\wedge \Delta^{\lambda r}\preccurlyeq x^{\lambda+1}\wedge \Delta^{\lambda r}=x^\lambda.$$  
where the last equality is due to the rigidity of~$x$.
\end{proof}

We now see that even without the hypothesis that $z$ is a prefix of some power of $x$, provided that $\lambda_x(z)$ is big enough, there is an initial segment of the left normal form of $z$ which consists of a power of $x$. 
\begin{proposition}\label{P:InitialSegment}
Let $z$ be a braid of infimum 0 and suppose that $\lambda=\lambda_x(z)\geqslant 2$.
Then the product of the $(\lambda-1)r$ first factors of the left normal form of $z$ is exactly  $x^{\lambda-1}$.
\end{proposition}

\begin{proof}
We may assume that $x^\lambda\neq z$, otherwise the result is trivial. So there exists a non-trivial positive $A$ so that $z=x^\lambda A$. Write $s_1\ldots s_r$ for the normal form of $x$. Let $1<j<r$ be the biggest integer so that $s_j$ has the form $\sigma_i^{-1}\Delta$ (see Observation~\ref{O:PropertiesOfx}(3)).
 From the algorithm for computing left normal forms -- see \cite{Gebhardt-GM}, Proposition~1 --, and because $x$ is rigid, it follows that the left normal form of $x^\lambda A$ starts with 
$x^{\lambda-1}s_1\ldots s_j$; otherwise $\inf(z)=0$ would be contradicted. 
\end{proof} 

Propositions \ref{P:BetweenPowers} and \ref{P:InitialSegment} admit analogues "on the right", namely if $z$ is a suffix of some positive power of $x$, then $z$ lies between two successive powers of $x$ with respect to the suffix order. Moreover, if $k\geqslant 2$ is the maximal integer so that $x^{k}$ is a suffix of $z$, then the left normal form of $z$ has a final segment consisting of $x^{k-1}$. However, these facts will not be used in the proof of Theorem \ref{T:InfDiam} so we do not prove them. Instead, we state the easier:

\begin{proposition}\label{P:FinalSegment}
Let $z$ be a positive braid of infimum 0 and assume that there is a positive integer $k$ so that $x^{k+1}\succcurlyeq z\succcurlyeq x^k$. Then the final segment of length $kr$ in the left normal form of $z$ consists of $x^k$.
\end{proposition}

\begin{proof}
Again, we may assume that $z$ is not a power of $x$. There exist by hypothesis some non-trivial positive braids $v$ and $w$ so that $x^{k+1}=wz$ and $z=vx^k$. Combining both, we get $x^{k+1}
=wvx^k$; cancelling $x^k$ on the right, it follows that $x=wv$, so that $v$ is a suffix of $x$. By Lemma \ref{L:PropertyOfx}(a), 
$z=vx^k$ is in left normal form as written. This shows the result.
\end{proof}
 
\begin{proposition}\label{P:PathGoesThroughx}
Suppose that $z_1,z_2$ are braids with infimum 0; let $v_i$ ($i=1,2$) be the vertex of $\mathcal C_{AL}$ whose distinguished representative is $z_i$. 
Let $\lambda_1=\lambda_x(z_1)$ and $\lambda_2=\lambda_x(z_2)$. Assume that $\lambda_2-\lambda_1\geqslant 3$. Then the path $A(v_1,v_2)$ contains $A(x^{\lambda_1+1}\Delta^{\Z},x^{\lambda_2-1}\Delta^{\Z})$. 
\end{proposition}
\begin{proof} See Figure~\ref{F:PfLemma}.
We look at the two paths $\gamma_1=A(v_1,x^{\lambda_2-1}\Delta^{\Z})$ and $\gamma_2=A(1,v_2)$. 

We claim that $\gamma_1$ and~$\gamma_2$ coincide along $A(x^{\lambda_1+1}\Delta^{\Z},x^{\lambda_2-1}\Delta^{\Z})$. 

Let us prove this claim.
On the one hand, by Proposition \ref{P:InitialSegment}, $A(1,x^{\lambda_2-1})$ is an initial segment of $\gamma_2$. 

On the other hand, let $z_3=z_1\wedge x^{\lambda_2-1}$ and let $v_3$ be the vertex of $\mathcal C_{AL}$ whose distinguished representative is~$z_3$.
By Lemma~\ref{L:pgcd}, $\gamma_1$ is the concatenation of $A(v_1,v_3)$ and $A(v_3,x^{\lambda_2-1}\Delta^{\mathbb Z})$. Note that $\lambda_x(z_3)=\lambda_1$. By Proposition~\ref{P:BetweenPowers}, 
$x^{\lambda_1}\preccurlyeq z_3\preccurlyeq x^{\lambda_1+1}$. It follows that 
$$x^{\lambda_2-\lambda_1-1}\succcurlyeq z_3^{-1}x^{\lambda_2-1} \succcurlyeq x^{\lambda_2-\lambda_1-2}.$$
By Proposition \ref{P:FinalSegment}, the left normal form of $z_3^{-1}x^{\lambda_2-1}$ terminates with $(\lambda_2-\lambda_1-2)r$ factors whose product is exactly $x^{\lambda_2-\lambda_1-2}$. In other words, $A(v_3,x^{\lambda_2-1}\Delta^{\mathbb Z})$ has a final segment equal to $A(x^{\lambda_1+1}\Delta^{\mathbb Z},x^{\lambda_2-1}\Delta^{\mathbb Z})$ and our claim is shown.

\begin{figure}[ht]
\centerline{\includegraphics{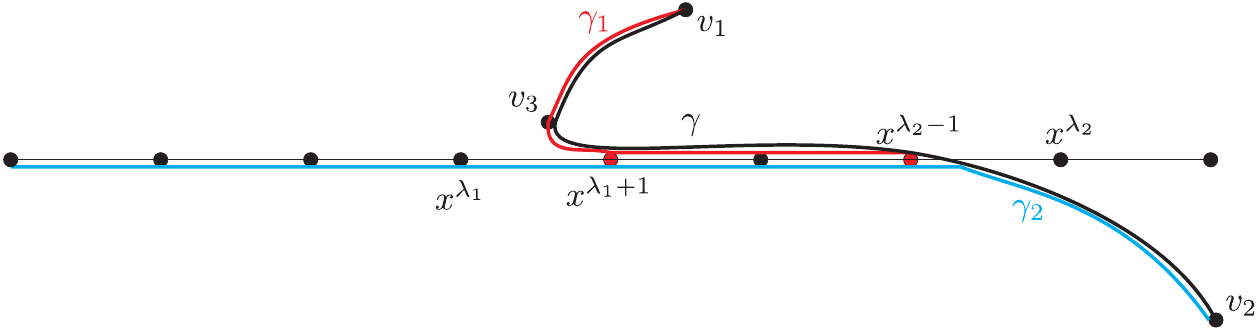}}
\caption{Proof of Proposition~\ref{P:PathGoesThroughx}.}
\label{F:PfLemma}
\end{figure}

Now consider the path $\gamma$ formed by the subpath of $\gamma_1$ between $v_1$ and $x^{\lambda_2-1}\Delta^{\Z}$, followed by the subpath of $\gamma_2$ between $x^{\lambda_2-1}\Delta^{\Z}$ and $v_2$. Observe that $\gamma$ connects $v_1$ and $v_2$ and that the product of the labels of the successive edges along $\gamma$ gives a left normal form. This says that $\gamma=A(v_1,v_2)$, hence showing the lemma. 
\end{proof}
%

We are now ready to complete the proof of Theorem~\ref{T:InfDiam}. We will do this by proving the following claim.

{\bf Claim } For all positive integers~$N$,  the $N$th power of $x$ lies at distance at least $\frac{N}{2}$ from the identity vertex in~$\mathcal C_{AL}$.

%

In order to prove this claim, we fix an $N$, and suppose for a contradiction that there exists a path of length $K$, with $K<\frac{N}{2}$, which connects the identity vertex with the vertex $x^N\Delta^{\Z}$ . Let $v_0=1, v_1,\ldots ,v_K=x^N\Delta^{\mathbb Z}$ be the vertices along this path. Notice that $\lambda_x(\underline{v_0})=0$ and $\lambda_x(\underline{v_K})=N$. Thus there is some integer~$i$ between 0 and $K-1$ such that $\lambda_x(\underline{v_{i+1}})\geqslant \lambda_x(\underline{v_i})+3$. This index~$i$ will play a key role in what follows.

%

By Proposition~\ref{P:PathGoesThroughx}, 
the path $A(v_i,v_{i+1})$ contains the subpath $$A(x^{\lambda_x(\underline{v_i})+1}\Delta^{\mathbb Z},x^{\lambda_x(\underline{v_{i+1}})-1}\Delta^{\Z}).$$
We will see that this contradicts the equality $d_{AL}(v_i,v_{i+1})=1$. The vertices $v_i$ and $v_{i+1}$ cannot be connected by an edge labeled by a simple element, otherwise $A(v_i,v_{i+1})$ would be a path of length 1 which cannot contain $A(x^{\lambda_x(\underline{v_i})+1}\Delta^{\mathbb Z},x^{\lambda_x(\underline{v_{i+1}})-1}\Delta^{\Z})$ as a subpath. Thus it only remains to see that there cannot be any absorbable element~$y$ so that $\underline{v_i} y\in \underline{v_{i+1}}\Delta^{\mathbb Z}$. 
Actually, such a braid~$y$ must have the form $y=\underline {v_i}^{-1}\underline{v_{i+1}}\Delta^k$ for some integer $k$. 
In order for $y$ to be absorbable, we must have $k=k_1=-\inf(\underline {v_i}^{-1}\underline{v_{i+1}})$  or $k=k_2=-\sup(\underline {v_i}^{-1}\underline{v_{i+1}})$. 
In the first case, $y$ is the braid whose left normal form is given reading the edges along $A(v_i,v_{i+1})$ and thus cannot be absorbable, by Lemma~\ref{L:subword}. 
In the second case, $y$ is negative; its inverse $y^{-1}$ is a positive braid whose left normal form is obtained reading the edges along the path $A(v_{i+1},v_i)$. Again, this braid cannot be absorbable because its left normal form contains $\partial x$ as a subword. This contradicts the choice that $d_{AL}(v_i,v_{i+1})=1$, completing the proof of the claim and of Theorem~\ref{T:InfDiam}.
\end{proof}


\subsection{Quasi-isometry with the curve complex}\label{SS:QiWithCC}

To conclude this paper, we turn to the question whether the additional length complex $\mathcal C_{AL}(B_n)$ is actually quasi-isometric to $\mathcal{CC}$, the curve complex of the $n$-times punctured disk. If the answer to this question is positive, as we conjecture, then our previous results show in particular that Garside normal forms in~$B_n$ project to unparametrized quasi-geodesics in $\mathcal{CC}$.

We shall construct a Lipschitz map
$$
\mathcal{CC} \longrightarrow \mathcal{C}_{AL}(B_n)
$$
The most natural way to think of this map is to introduce first another model for the curve complex. Start with the Cayley graph of~$B_n$, with respect to any finite generating set, for instance Garside's. Next we recall that there are only finitely many \emph{round} simple closed curves in~$D_n$, the disk with $n$ punctures lined up horizontally. For each such curve~$c$, look at the set $S_c\subset B_n$ consisting of all braids which stabilise~$c$, and build a cone on the subset $S_c$ of the Cayley graph, i.e.\ introduce a new vertex and connect each element of~$S_c$ by an edge of length one to this new vertex. Also build copies of these finitely many cones all over the Cayley graph by translating them using the left action of $B_n$ on the Cayley graph. Let us denote the resulting space $\MMmodel$; this is sometimes called the ``electric space". 

As proven by Masur and Minsky~\cite[Lemma 3.2]{MM1}, the space $\MMmodel$ is quasi-isometric to the curve complex of $B_n$; a quasi-isometry $\MMmodel\to \mathcal{CC}$ is given by sending the vertex of the Cayley graph corresponding to $x\in B_n$ to the curve $x.c_0$, where $c_0$ is any simple closed curve in $D_n$ (e.g.\ a round one). 

Now we can construct a very nice map $\phi\co \MMmodel \to \mathcal C_{AL}$. It suffices to map vertices and edges of $\MMmodel$ belonging to the Cayley graph by the identity map. Every cone vertex is mapped in the same way as an arbitrarily chosen one of its adjacent vertices. Finally, every cone edge can be sent to an arbitrarily chosen edge path in $\mathcal C_{AL}$ of length at most nine (this is possible by Lemma~\ref{L:ReducibleSmallDistance}).

\begin{conjecture}\label{C:QI} 
The map $\phi\co \MMmodel \to \mathcal C_{AL}(B_n)$ is a quasi-isometry.
\end{conjecture}

All that remains to be proven is that the map~$\phi$ does not shrink distances too much. More precisely, it suffices to prove that there exists a positive number $D$ with the following property: if $x\in B_n$ is such that $d_{AL}(1_G,x)=1$, then $d_{\MMmodel}(1_G, x)\leqslant D$. This comes down to the very plausible claim that every absorbable braid is the product of at most~$D$ braids fixing some round curve.

{\bf{Acknowledgements.}} 
Support by ANR LAM (ANR-10-JCJC-0110) is gratefully acknowledged.  
The first author is supported by the ``initiation to research'' project no.11140090 from Fondecyt, by MTM2010-19355 and FEDER, and Fondecyt Anillo 1103.


\begin{thebibliography}{99}
\bibitem{BGN} {\bf D.~Benardete, M.~Gutierrez, Z~Nitecki}, {\it A combinatorial approach to reducibility of mapping classes}, Mapping class groups and moduli spaces of Riemann surfaces (G\"ottingen, 1991/Seattle, WA, 1991),  1--31,
Contemp. Math., 150, Amer. Math. Soc., Providence, RI, 1993.

\bibitem{BestvinaFujiwara02} {\bf M.~Bestvina, K.~Fujiwara}, {\it Bounded cohomology of subgroups of mapping class groups}, Geom. Topol. 6 (2002), 69--89

\bibitem{B-G-GM} {\bf J.~Birman, V.~Gebhardt, J.~Gonz\'alez-Meneses}, {\it Conjugacy in Garside groups~I: cyclings, powers and rigidity}, Groups Geom. Dyn. 1 (2007), 221--279

\bibitem{Bowditch} {\bf B.~Bowditch}, {\it Uniform hyperbolicity of the curve graph}, Pacific J. Math, 269 (2014), 269--280.

\bibitem{Bowditch08} {\bf B.~Bowditch}, {\it Tight geodesics in the curve complex}, Invent.\ Math.\ 171 (2008), no. 2, 281--300. 

\bibitem{CalvezStandard} {\bf M.~Calvez}, {\it Dual Garside structure and reducibility of braids}, J.\ Algebra 356 (2012), 355--373.

\bibitem{CarusoGeneric1} {\bf S.~Caruso}, {\it On the genericity of pseudo-Anosov braids I: rigid braids},  arXiv:1306.3757

\bibitem{CarusoWiestGeneric2} {\bf S.~Caruso, B.~Wiest}, {\it On the genericity of pseudo-Anosov braids II: conjugations to rigid braids},  arXiv:1309.6137

\bibitem{Charney} {\bf R. Charney}, {\it Geodesic automation and growth functions for Artin  groups o finite type}, Math. Ann. 301 (1995), no. 2, 307--324.



\bibitem{CMW} {\bf R.~Charney, J.~Meier, K.~Whittlesey}, {\it Bestvina's normal form complex and the homology of Garside groups},  Geom. Dedicata 105 (2004), 171--188.

\bibitem{Elrifai-Morton} {\bf E. ElRifai, H. Morton}, {\it Algorithms for positive braids}, Quart.\ J.\ Math.\ Oxford Ser.~(2)  45 (1994), 479--497.

\bibitem{DehornoyGarside} {\bf P. Dehornoy}, {\it Groupes de Garside}, Ann. Sci. \'Ecole Norm. Sup. (4) 35 (2002), 267--306.

\bibitem{GarsideFoundations} {\bf P.~Dehornoy, with F.~Digne, E.~Godelle, D.~Krammer, and J.~Michel}, {\it Foundations of Garside Theory}, EMS Tracts in Mathematics, volume 22, European Mathematical Society, 2015.

\bibitem{DehornoyParis} {\bf P. Dehornoy, L. Paris}, {\it Gaussian groups and Garside groups: two generalizations of Artin groups}, Proc. London Math. Soc., 79 (1999), 569--604.

\bibitem{Thurston} {\bf D.B.A. Epstein, J. Cannon, D. Holt, S. Levy, M. Paterson, W. Thurston}, \textit{Word processing in groups}, Jones and Bartlett Publishers, Boston, MA, 1992.

\bibitem{Gebhardt-GM} {\bf V.~Gebhardt, J.~Gonz\'alez-Meneses}, {\it The cyclic sliding operation in Garside groups}, Math. Z. 265 (2010), no. 1, 85--114.

\bibitem{GonzalezMenesesRed} {\bf J.~Gonz\'alez-Meneses}, {\it On reduction curves and Garside properties of braids}, Topology of algebraic varieties and singularities, 227--244, Contemp. Math., 538, Amer. Math. Soc., Providence, RI, 2011.

\bibitem{GonzalezMenesesWiest1} {\bf J.~Gonz\'alez-Meneses, B.~Wiest}, {\it On the structure of the centralizer of a braid}, Ann. Sci. \'Ecole Norm. Sup. (4) 37 (2004), no. 5, 729--757.

\bibitem{GonzalezMenesesWiest2} {\bf J.~Gonz\'alez-Meneses, B.~Wiest}, {\it Reducible braids and Garside theory}, Algebr. Geom. Topol. 11 (2011), no. 5, 2971--3010.

\bibitem{HPW} {\bf S.~Hensel, P.~Przytycki, R.C.H.~Webb}, {Slim unicorns and uniform hyperbolicity for arc graphs and curve graphs}, arXiv:1301.5577.

\bibitem{Ivanov} {\bf  N.V. Ivanov}, {\it Automorphism of complexes of curves and of Teichm\"uller spaces}, Internat. Math. Res. Notices 1997, no. 14, 651--666.

\bibitem{Maher} {\bf J.~Maher}, {\it Exponential decay in the mapping class group}, J. Lond. Math. Soc. (2) 86 (2012), no. 2, 366--386.

\bibitem{MM1} {\bf H.~Masur, Y.~Minsky}, {\it Geometry of the complex of curves I: Hyperbolicity}, Invent. math. 138 (1999), 103--149.

\bibitem{MM2} {\bf H.~Masur, Y.~Minsky}, {\it Geometry of the complex of curves II: Hierarchical structure}, Geometric and Functional Analysis, 10(4)(2000), 902--974 .

\bibitem{MasurSchleimer} {\bf H.~Masur, S.~Schleimer}, {\it The geometry of the disk complex}, Journal of the AMS 26~(1), (2013), 1--62.

\bibitem{PrzSisto} {\bf P.~Przytycki, A.~Sisto}, {\it A note on acylindrical hyperbolicity of mapping class groups}, arXiv:1502:02176

\bibitem{SistoGeneric} {\bf A.~Sisto}, {\it Contracting elements and random walks},  arXiv:1112.2666

\bibitem{SchleimerWiest} {\bf S.~Schleimer, B.~Wiest}, unpublished, available on request. 

\bibitem{WAutomLoxGeneric} {\bf B.~Wiest}, {\it On the genericity of loxodromic actions},  arXiv:1406.7041

\end{thebibliography}
\end{document}